\documentclass[12pt]{article}
\usepackage{amsmath,amsthm}
\usepackage{amsfonts,mathrsfs}
\usepackage{color}
\allowdisplaybreaks
\theoremstyle{plain}
\newtheorem{thm}{\noindent\bf Theorem}
\newtheorem{cor}{\noindent\bf Corollary}
\newtheorem{lem}{\noindent\bf Lemma}
\newtheorem{prop}{\noindent\bf Proposition}
\theoremstyle{remark}
\newtheorem{rmk}{\noindent \bf \textit{Remark}}
\newtheorem{exam}{\noindent \bf Example}
\newcommand{\D}{\displaystyle}
\newcommand{\non}{\nonumber}
\newcommand{\Pm}{\mathbb{P}}
\newcommand{\Em}{\mathbb{E}}

\usepackage{ulem}

\newcommand{\kb}{\kappa_{b}^{+}}
\newcommand{\kc}{\kappa_{c}^{-}}
\newcommand{\ke}{\kappa^{\{d\}}}
\newcommand{\wz}{\mathbb{W}}
\newcommand{\zz}{\mathbb{Z}}
\newcommand{\Rr}{V}
\newcommand{\Ror}{V^{(\omega)}}

\newcommand{\wor}{w^{(\omega)}}
\newcommand{\zor}{z^{(\omega)}}
\newcommand{\wqr}{w^{(q)}}
\newcommand{\wpr}{w^{(p)}}
\newcommand{\zqr}{z^{(q)}}
\newcommand{\zpr}{z^{(p)}}

\newcommand{\wo}{W^{(\omega)}}
\newcommand{\zo}{Z^{(\omega)}}
\newcommand{\wzo}{\mathbb{W}^{(\omega)}}

\newcommand{\asumref}[1]{{\textbf{(H)}}}

\begin{document}

\title{On  weighted occupation times for refracted spectrally negative L\'evy processes}
\author{Bo Li and Xiaowen Zhou\\
School of Mathematics and LPMC, Nankai University, Tianjin, P.R.China\\
Department of Mathematics and Statistics, Concordia University, Canada}
\maketitle


\begin{abstract}
For refracted spectrally negative L\'evy processes, we identify expressions of several quantities related to Laplace transforms on their weighted occupation times until first exit times. Such quantities are expressed in terms of unique solutions to integral equations involving   weight functions and scale functions for the associated spectrally negative L\'evy processes. Previous results on refracted L\'evy processes are recovered.
\end{abstract}

\section{Introduction}
In this paper,
we are interested in evaluating the $\omega$-weighted occupation time of refracted spectrally negative L\'evy process,
which is formally defined in \cite{Kyprianou2010:Rlevy} as the unique solution to stochastic differential equation
\begin{equation}\label{defn:x}
dX_{t}= dY_{t}- \delta 1_{\{X_{t}\geq a\}} dt=dZ_{t}+ \delta 1_{\{X_{t}< a\}} dt,
\end{equation}
where $Y=(Y_{t})_{t\geq0}$ is a spectrally negative L\'evy process (SNLP in short) and
$Z_{t}=Y_{t}-\delta t, \,\, t\geq 0$.

A motivation of studying refracted process stems from its
applications in stochastic control. In many insurance risk models,
see for example
\cite{Kyprianou2012:optimal:threshold,Hernandez-Hernandez2016:optimal,Cheung2016:optimal}
and references therein, to maximise the amount of discounted
dividends paid up to time of ruin, among all admissible control
strategies the optimal dividend strategy is either paying nothing or
paying dividends as much as possible in the so-called solvency
regions.
%
When a constant ceiling $\delta$ is imposed for the dividend rate,
under  further conditions the optimal policy reduces to the so-called threshold dividend strategy and the surplus process with dividends becomes  the refracted process, in which
the insurance company  pays nothing when the reserve is below a certain critical level,
and pays dividends at the maximal  rate $\delta$ when the reserve is above the level.

To our best knowledge,
the refracted L\'evy process is first introduced in  \cite{Jeanblanc1995, Asmussen1997:optimal} for Brownian risk mode,
where by making use of HJB functional equation, the threshold strategy is shown to be optimal if the dividend rate is bounded above by some constant.
The analogous problem for Cr\'emer-Lundberg risk model with exponential claim size distribution is studied in \cite{Gerber2006:naaj}.
Some actuarial quantities of risk model with threshold dividend are investigated in \cite{Lin2006:ime, Wan2007:ime}.
It is studied under the framework of spectrally negative L\'evy process in
\cite{Kyprianou2012:optimal:threshold} where a sufficient condition on L\'evy measure is found under which  the threshold strategy is optimal.
\cite{Kyprianou2010:Rlevy} focuses on the existence and uniqueness of solution to \eqref{defn:x} and some fluctuation identities for $X$ are also established.
\cite{Kyprianou2013:occupationtime:Rlevy} investigats the occupation times of half lines,  \cite{Renaud2014:Rlevy} identifies the distribution of various functionals related,
and \cite{Zhou2015:Rlevy, Wu2016:Rlevy} mainly consider general L\'evy process with rational jumps.

During the last several years there have been a series of papers
concerning  occupation time related problems for SNLP. These
problems arise from both theoretical interests and  the applications
in risk theory and finance; see for example \cite{Gerber2012:omega,
Landriault2011:occupationtime:levy,
Loeffen2014:occupationtime:levy,Zhou2014:occupationtime:levy,Zhou2015:occupationtime:levy,BoLi:sub1,Kyprianou2013:occupationtime:Rlevy,Renaud2014:Rlevy,Zhou2015:Rlevy,Wu2016:Rlevy}.
Among them using a perturbation approach,
\cite{Landriault2011:occupationtime:levy} studies the occupation
times of semi-infinite intervals. For the occupation times spent in a certain
interval, using a strong approximation approach
\cite{Loeffen2014:occupationtime:levy} identifies joint Laplace
transforms until first passage times.  Laplace transforms involving
joint occupation times are investigated in
\cite{Zhou2014:occupationtime:levy,Zhou2015:occupationtime:levy}
with a  Poisson approach. Expressions involving occupation time over
a finite interval and resolvent measure are found in
\cite{Guerin2014:occupationmeasure:levy}. The results are typically
expressed using scale functions for SNLP. By further improving the Poisson
approach,    fluctuation identities on  weighted
occupation times for SNLP are obtained in \cite{BoLi:sub1} which generalize
many of the previous results. In \cite{BoLi:sub1} the results are
expressed in terms of unique solutions to integral equations specified using the
scale functions and the weight function for occupation time.

Given the previous results on refracted SNLP and on occupation times
for SNLP,  our goal in this paper is  to establish identities concerning the Laplace
transform of
\begin{equation}\label{defn:L}
L(t):=\int_{0}^{t}\omega(X_{s}) ds
\end{equation}
for  a locally bounded measurable nonnegative function
$\omega(\cdot)$   on $\mathbb{R}$, which is called a
$\omega$-weighted occupation time in \cite{BoLi:sub1}. Such
identities can also be treated as identities for refracted SNLP
killed at an occupation time dependent rate.  To prove the main results, we adapt the previous approach of
\cite{BoLi:sub1} by replacing the  Poisson approach there with a Feyman-Kac type argument.

Our main results are expressed in terms of functions $(\wor,\zor)$,
which extends the notations introduced in \cite{Renaud2014:Rlevy}
and depends on the classical scale function $(W,Z)$ as well as the
weight function $\omega$. For some simpler examples of weight
function, we could find more explicit expressions of $(\wor,\zor)$
in terms of $(W,Z)$ and recover previous results in
\cite{Kyprianou2010:Rlevy,Kyprianou2013:occupationtime:Rlevy,Renaud2014:Rlevy}.

The remainder of the paper is organised as follows. After the review of previous work on refracted SNLP and occupation times for SNLP and a summary of the main results, in Section \ref{sec:2} we present preliminary results on SNLP, scale functions and exit problems for refracted SNLP. Section \ref{sec:3} contains  the main results whose proofs are deferred to Section \ref{sec:5}. More detailed discussions are carried out for examples in Section \ref{sec:4}.


\section{Preliminaries}\label{sec:2}
We first briefly review the spectrally negative L\'{e}vy processes, the associated scale functions and some known results.
For further details, we refer the readers to  Bertoin \cite{Bertoin96:book} and Kyprianou \cite{Kyprianou2014:book:levy}.
Throughout the paper,
$Y$ denotes an SNLP,
$Z_{t}=Y_{t}-\delta t$
and $X$ is the unique solution to \eqref{defn:x} called refracted SNLP. The law of $X$ for $X_{0}=x$ is denoted by $\Pm_{x}$ and the corresponding expectation by $\Em_{x}$.
We write $\Pm$ and $\Em$ when $x=0$.

Let $Y=(Y_{t})_{t\geq0}$ be a spectrally negative L\'{e}vy process,
that is a stochastic process with stationary and independent increments and without positive jumps.
Its Laplace transform exists and is specified by
\[\Em(\exp(\theta Y_{t}))=\exp(\psi(\theta) t), \quad\forall \theta\geq0.\]
Function $\psi(\theta)$, known as the Laplace exponent of $Y$, is continuous, strictly convex on $\mathbb{R}^{+}$ and given by the L\'{e}vy-Khintchine formula:
\[
    \psi(\theta)=\frac{\sigma^{2}}{2} \theta^{2}+ \gamma \theta
    +\int_{(0,\infty)}(e^{-\theta x}-1+\theta
    x   \boldsymbol{1}_{\{x<1\}}) \Pi (d x),
\]
where $\gamma\in\mathbb{R}$, $\sigma\geq0$ and the L\'{e}vy measure $\Pi$ is a $\sigma$-finite measure on $(0,\infty)$ such that $\int_{\mathbb{R}^{+}} (1\wedge x^{2}) \Pi (dx)< \infty$.
For $q\geq0$, the $q$-scale function is defined as a continuous and increasing function such that $W^{(q)}(x)=0$ for $x< 0$ and
\[
\int_{0}^{\infty} e^{-sy} W^{(q)}(y) dy=\frac{1}{\psi(s)-q}, \quad\text{for $s> \Phi(q)$},
\]
where $\Phi(q):=\sup\{s\geq 0, \psi(s)=q\}$ denotes the right inverse of $\psi(\cdot)$. With  the first scale function $W^{(q)}(\cdot)$, we can define another  scale function by
\begin{equation}\label{Z-W}
Z^{(q)}(x):=1+ q \int_{0}^{x}W^{(q)}(y) dy, \quad\text{for $x\in\mathbb{R}$}.
\end{equation}
We write $W(x)=W^{(0)}(x)$ and $Z(x)=Z^{(0)}(x)$ when $q=0$.
It is known that, for $q>0$
\begin{equation}\label{eqn:limit:wz}
\frac{W^{(q)}(x-a)}{W^{(q)}(x)}\to e^{-\Phi(q)a},
\quad
e^{-\Phi(q)x}W^{(q)}(x)\to\Phi'(q)
\quad \text{and}\quad
\frac{Z^{(q)}(x)}{W^{(q)}(x)}\to \frac{q}{\Phi(q)},
\end{equation}
as $x\to\infty$.
Since $\log W(x)$ is concave on $\mathbb{R}^{+}$, it is absolutely continuous with respect to Lesbegue measure and $W'(x)$ is well defined a.e.
Similar conclusion can be derived for every $W^{(q)}(x)$ by change of measure.
We refer to \cite{Kuznetsov2012:scalefunction, Hubalek2011:scalefunction:examples} for more detailed discussions and examples of scale functions.

For process $Z$, we denote by
$\psi_{Z}(\theta)$ its Laplace exponent,
$\varphi(q):=\sup\{s\geq0, \psi_{Z}(s)=q\}$ the right inverse of $\psi_{Z}(\cdot)$,
and
$(\wz^{(q)}, \zz^{(q)})$ the $q$-scale functions  associated with $Z$.
It can be checked directly that,
$\psi_{Z}(s)=\psi(s)- \delta s$, $\varphi(q)>\Phi(q)$ and
\begin{equation}\label{eqn:wzw}
\wz(x)=W(x)+ \delta \int_{0-}^{x} \wz(x-z)W(dz),
\end{equation}
where $W(dz)$ stands for the  Stieltjes integral on $\mathbb{R}$ induced by function $W$ and
$W(\{0\})=W(0)=\wz(0)(1-\delta W(0))$.

\begin{rmk}
We remark that, for any locally bounded measurable function $f\geq 0$,
its Stieltjes integral with respect to $m(\cdot)$, denoted
\[
\int_{a}^{b} f(z)m(dz)= \int_{(a,b]} f(z) m(dz)=\int_{0}^{b-a} f(z+a)m(dz+a)
\]
is the integral of $f$ on the interval $(a,b]$ and that $\int_{a}^{b} m(dz)=m(b)-m(a)$.
If the integral is on $[a,b]$,
we write $\int_{a-}^{b}f(z)m(dz)=\int_{[a,b]}f(z)m(dz)$ as in \eqref{eqn:wzw}.
\end{rmk}

The following hypothesis on $Y$ is introduced in \cite{Kyprianou2010:Rlevy}.
\begin{description}\label{H}
\item[(H)] The constant $0<\delta< \gamma+ \int_{(0,1)} x \Pi(dx)$ if $Y$ has paths of bounded variation.
\end{description}
It is shown in \cite[Theorem 1]{Kyprianou2010:Rlevy} that under hypothesis \textbf{(H)},
equation \eqref{defn:x} has a unique strong solution.
Actually, for the case of bounded variation, we have $W(0)=\frac{1}{\gamma+ \int_{(0,1)} x \Pi(dx)}$. Thus,  condition \textbf{(H)} is equivalent to the following assumption:
\begin{description}
\item[(H')]   $1-\delta W(0)>0$, i.e.
 $Z$ is not the negative of a subordinator,
\end{description}
which will be in force throughout the remainder of the paper.

We denote by
\[
\kb:=\inf\{t>0, X_{t}>b\}\quad\text{and}\quad \kc:=\inf\{t>0, X_{t}<c\},
\]
with convention $\inf\emptyset=\infty$,
the first passage times of $X$
and for any $x, y\in\mathbb{R}$ define the following auxiliary function,
\begin{align}
w(x,y):=&\ W(x-y)+ \delta \int_{a}^{x} \wz(x-z)W(dz-y)\label{defn:wr:1}\\
=&\ \wz(x-y)- \delta \int_{y-}^{a} \wz(x-z)W(dz-y) \label{defn:wr:2}
\end{align}
in light of \eqref{eqn:wzw}.
For  $a>0$, the equation above reduces to
\begin{align*}
w(x,0)=&\ W(x)+ \delta \int_{a}^{x} \wz(x-z)W^{\prime}(z) dz
\end{align*}
and
\begin{align*}
 w(x,y)=&\ \left\{\begin{array}{l@{\quad \text{for}\quad}l}
W(x-y)+  \delta \int_{a}^{x} \wz(x-z)W^{\prime}(z-y) dz & y\in(0,a]\\
\wz(x-y) & y\in(a,\infty)
\end{array}\right.
\end{align*}
which are functions used in \cite{Kyprianou2010:Rlevy} for
{ conclusion on} interval $(0,b)$,
and the following results follow from Theorems 4 and 6 in \cite{Kyprianou2010:Rlevy} after a spatial shifting argument.
\begin{prop}\label{prop:Rlevy}
For any $x,c,b$ such that $a, x\in(c,b)$, we have
\begin{equation}\label{eqn:ky2010:prob}
\Pm_{x}\left(\kb<\kappa_{c}^{-}\right)= \frac{w(x,c)}{w(b,c)},
\end{equation}
and the resolvent measure $V$ of $X$ killed at exiting $[c,b]$ is given by
\begin{align}\label{eqn:ky2010:reso}
\Rr f(x):=&\ \int_{0}^{\infty} \Em_{x}[f(X_{t}); t\leq \kb\wedge\kappa_{c}^{-}] dt\non\\
=&\ \int_{c}^{b} f(y) \left(\frac{w(x,c)}{w(b,c)}w(b,y)- w(x,y) \right) dy
\end{align}
for any bounded measurable function $f$ on $[c, b]$.
\end{prop}

\begin{rmk}
Notice that for $c=0<x<b$, a version of the density for measure $V(x,dy)$ on $[0,b]$ determined by \eqref{eqn:ky2010:reso} is
\[
\frac{w(x,0)}{w(b,0)}w(b,y)- w(x,y), \quad\text{for\ } 0<y<b,\]
and its value at $y=a$ is
\[
\frac{w(x,0)}{w(b,0)}w(b,a)- w(x,a)
=
\frac{w(x,0)}{w(b,0)}W(b-a)-W(x-a),
\]
which is different from the value at $y=a$ of the the corresponding density in \cite[Theorem 6]{Kyprianou2010:Rlevy}  given by
\[\frac{w(x,0)}{w(b,0)}\wz(b-a)-\wz(x-a).\]
Otherwise, the density from (\ref{eqn:ky2010:reso}) agrees with that in \cite[Theorem 6]{Kyprianou2010:Rlevy},
and clearly they are associated with the same absolutely continuous measure on $(0, b)$.
\end{rmk}


\section{Main results}\label{sec:3}
Before stating our main results, we  introduce two more generalized scale functions
$(\wor, \zor)$ which, for the  $w(x,y)$ defined in \eqref{defn:wr:1},   are solutions to
\begin{align}
\wor(x,y)=&\ w(x,y)+ \int_{y}^{x} w(x,z)\omega(z) \wor(z,y) dz
\label{defn:wor}
\end{align}
and
\begin{align}
\zor(x,y)=&\ 1+ \int_{y}^{x} w(x,z)\omega(z) \zor(z,y) dz,
\label{defn:zor}
\end{align}
respectively. The existence and uniqueness of solutions to (\ref{defn:wor}) and (\ref{defn:zor}) are assured by Lemma \ref{lem:equation}.
It can be found that they  play similar roles as $(W,Z)$  in the fluctuation theory of SNLP.

\begin{lem}\label{lem:equation}
Let $h(x,y)$ be a locally bounded function on $\mathbb{R}^{2}$. Equation
\begin{equation}\label{eqn:lem}
H^{(\omega)}(x,y)= h(x,y)+ \int_{y}^{x} w(x,z)\omega(z) H^{(\omega)}(z,y) dz,
\end{equation}
admits a unique locally bounded solution on $\mathbb{R}^{2}$ satisfying $H^{(\omega)}(x,y)=h(x,y)$ for $x\leq y$.
\end{lem}

Given $\wor(x,y)$ and Lemma \ref{lem:equation}, we have the following remark.

\begin{rmk}\label{rmk:equation}
Let $\nu(dy)$ be a Radon measure on $\mathbb{R}$. By Fubini's Theorem one can check directly that  $ \left(\wor \nu\right)(x,y):=\int_{y}^{x} \wor(x,z)\nu(dz)$ satisfies
\[
\left(\wor\nu\right)(x,y)= \int_{y}^{x} w(x,z)\nu(dz) + \int_{y}^{x} w(x,z)\omega(z) \left(\wor\nu\right)(z,y) dz,
\]
and is thus the unique solution to this equation.
In particular, for any $f\geq0$ locally bounded, $\wor f(x,y):=\int_{y}^{x} \wor(x,z) f(z) dz$ is the solution to
\[
\wor f(x,y)= \int_{y}^{x}w(x,z)f(z) dz+ \int_{y}^{x} w(x,z) \omega(z) \left(\wor f\right)(z,y) dz.
\]
\end{rmk}

We now state the main results.

\begin{thm}\label{thm:1} Given $c<b$, we have for $x\in[c,b]$
\begin{equation}
\Em_{x}\left[e^{-L(\kb)}; \kb\leq \kc\right]
= \frac{\wor(x,c)}{\wor(b,c)}\label{thm:1a}
\end{equation}
and
\[
\Em_{x}\left[e^{-L(\kc)}; \kc\leq \kb\right]
= \zor(x,c)- \frac{\wor(x,c)}{\wor(b,c)}\zor(b,c).
\]
For any $x, y\in (c, b)$, an expression of the resolvent of $X$ killed at exiting $[c,b]$ is given by
\begin{align*}
\Ror(x,dy):=&\ \int_{0}^{\infty}\Em_{x}\left(e^{-L(t)}; X_{t}\in dy, t\leq \kb\wedge\kc\right)\,dt\\
=&\ \left(\frac{\wor(x,c)}{\wor(b,c)}\wor(b,y)- \wor(x,y)\right) dy.
\end{align*}
\end{thm}

The conclusions above are similar to the corresponding results on $\omega$-weighted occupation problem for  SNLP in \cite{BoLi:sub1}, where the auxiliary functions $(\wo,\zo)$ are defined as the unique solution,  respectively, to the following equations.
\begin{align}
\wo(x,y)=&\ W(x-y)+ \int_{y}^{x} W(x-z)\omega(z) \wo(z,y) dz,\label{defn:wo}\\
\zo(x,y)=&\ 1+ \int_{y}^{x} W(x-z)\omega(z) \zo(z,y) dz.\label{defn:zo}
\end{align}
Therefore, we  present the following relation between them,
which generalises \eqref{defn:wr:1} and \eqref{defn:wr:2}  since for  $\omega(\cdot)\equiv0$, $W^{(\omega)}(x,y)=W(x-y)$ by definition.
\begin{prop}\label{prop:wor:wo} For $x\geq y$, we have
\begin{equation}\label{eqn:wor:wo}
\begin{split}
\wor(x,y)=&\ \wo(x,y)+ \delta \int_{a}^{x} \wz^{(\omega)}(x,z)W^{(\omega)}(dz,y)\\
=&\ \wz^{(\omega)}(x,y)- \delta \int_{y-}^{a} \wz^{(\omega)}(x,z) W^{(\omega)}(dz,y)\\
\end{split}
\end{equation}
and
\begin{equation}\label{eqn:zor:zo}
\begin{split}
\zor(x,y)=&\ \zo(x,y)+ \delta \int_{a}^{x} \wz^{(\omega)}(x,z)Z^{(\omega)}(dz,y)\\
=&\ \zz^{(\omega)}(x,y)- \delta \int_{y}^{a} \wz^{(\omega)}(x,z) Z^{(\omega)}(dz,y).
\end{split}
\end{equation}
\end{prop}

From \eqref{defn:wo} and \eqref{defn:zo}, for every $y\in\mathbb{R}$, it can be shown  that $\zo(\cdot, y)$ and $\wo(\cdot,y)$ are increasing functions.
The associated Stieltjes measures satisfy  respectively,
\begin{equation}\label{eqn:wzo:stj}
\left\{\begin{split}
\wo(dx,y)=&\ W(dx-y)+ \int_{\mathbb{R}} W(dx-z)\omega(z)\wo(z,y) dz,\\
\zo(dx,y)=&\ \int_{\mathbb{R}} W(dx-z)\omega(z) \wo(z,y) dz,
\end{split}\right..
\end{equation}
with $\wo(\{y\},y)=W(0)$ and $\zo(\{y\},y)=0$,
noticing that $0=W(u-v)=\wo(u,v)$ for $u<v$.
At the refraction point $a$, we have
$\wor(a,a)=W(0)=\wz(0)(1-\delta W(0))$ and
\begin{itemize}
\item for  $a>x$, $\wor(x,y)=\wo(x,y)$ and $\zor(x,y)=\zo(x,y)$,
\item for $y>a$, $\wor(x,y)=\wzo(x,y)$ and $\zor(x,y)=\zz^{(\omega)}(x,y)$,
\item for $x=a> y$, $\wor(a,y)= \wo(a,y)$ and $\zor(a,y)= \zo(a,y)$,
\item for $x> a=y$, $\wor(x,a)=\wzo(x,a)(1-\delta W(0))$,
$\zor(x,a)= \zz^{(\omega)}(x,a)$.
\end{itemize}
In particular, $\wor(x,y)$ is continuous at $(a,a)$ if and only if $W(0)=0$.

We also have the following scale function identities similar to those associated to $(W^{(q)},Z^{(q)})$.

\begin{prop}\label{prop:idens:2} Let $(w^{(\omega_{1})},z^{(\omega_{1})})$ and $(w^{(\omega_{2})},z^{(\omega_{2})})$ be the generalized scale functions with respect to weight functions $\omega_{1}(\cdot)\geq0$ and $\omega_{2}(\cdot)\geq0$, respectively.
Then for $x,y\in\mathbb{R}$
\begin{align}
w^{(\omega_{2})}(x,y)-w^{(\omega_{1})}(x,y)=&\ \int_{y}^{x} w^{(\omega_{1})}(x,z)(\omega_{2}(z)-\omega_{1}(z))w^{(\omega_{2})}(z,y) dz
\end{align}
and
\begin{align}
z^{(\omega_{2})}(x,y)-z^{(\omega_{1})}(x,y)=&\ \int_{y}^{x} w^{(\omega_{1})}(x,z)(\omega_{2}(z)-\omega_{1}(z))z^{(\omega_{2})}(z,y) dz.
\end{align}
\end{prop}

\begin{cor}[First hitting time]\label{cor:1}
 For any $d\in(c,b)$, let $\ke:=\inf\{t>0, X_{t}=d\}$ be the first hitting time. We have for $x\in[c,b]$
\begin{equation}
\Em_{x}\left[e^{-L(\ke)}; \ke\leq \kb\wedge\kc\right]
=\frac{\wor(x,c)}{\wor(d,c)}- \frac{\wor(x,d)}{\wor(b,d)}\frac{\wor(b,c)}{\wor(d,c)}.
\end{equation}
\end{cor}
\begin{proof}[Proof of Corollary \ref{cor:1}] Observing that due to  absence of positive jumps,
 \[ \{\kappa_{d}^{-}\leq \kb<\infty\}=\{\ke\leq \kb<\infty\} \,\,\, \Pm_{x}\text{-a.s.}\]
  Therefore, for $x\in[c,b]$
\begin{align*}
&\ \Em_{x}\left[e^{-L(\kb)}; \kb\leq \kc\right]\\
=&\ \Em_{x}\left[e^{-L(\kb)}; \kb\leq \kc, \kb\leq \kappa_{d}^{-}\right]+
\Em_{x}\left[e^{-L(\kb)}; \kb\leq \kc, \kappa_{d}^{-}\leq \kb\right]\\
=&\ \Em_{x}\left[e^{-L(\kb)}; \kb\leq  \kappa_{d}^{-}\right]+ \Em_{x}\left[e^{-L(\kb)}; \ke\leq \kb\leq \kc\right]\\
=&\ \Em_{x}\left[e^{-L(\kb)}; \kb\leq  \kappa_{d}^{-}\right]+ \Em_{x}\left[e^{-L(\ke)}; \ke\leq \kb\wedge\kc\right] \Em_{d}\left[e^{-L(\kb)}; \kb\leq \kc\right].
\end{align*}
The desired identity then follows from the Markov property and Theorem \ref{thm:1}.
\end{proof}

Indeed, given the resolvent measure, appealing to compensation formula
from \cite[O.5]{Bertoin96:book}
we have the joint distribution:
\[\Em_{x}\left(e^{-L(\kc)}; \kc\leq \kb, X(\kc-)\in dy, X(\kc)\in dz\right)
=\Ror(x,dy) (-\Pi(y-dz))\] for $b>y>c\geq z$.
Since $\{x, \Pi(\{x\})>0\}$ is at most countable
and $V(x,dy)$ is a continuous measure,
the creeping can not be caused by a jump.
Therefore, similar to the case of a L\'evy process,
a refracted SNLP creeps downward at a lower level with positive probability
only if it has a nontrivial Gaussian part.
We  restrict ourselves to $\sigma>0$ in the following corollary. $W(\cdot)$ and $\wor(x,x-\cdot)$ are thus continuously differentiable on $(0,\infty)$ from \cite[Theorem 1]{Chan2011:scalefunction:smooth}, $W(0)=0$ and $W'(0+)={2}/{\sigma^{2}}$.

\begin{cor}[Creeping]\label{cor:2} For $d<x<b$ and $d\leq a$ we have
\begin{equation}
\Em_{x}\left[e^{-L(\kappa^{-}_{d})}; X(\kappa^{-}_{d})=d, \kappa^{-}_{d}\leq \kb\right]
= \frac{\sigma^{2}}{2} \left(\frac{\wor(x,d)}{\wor(b,d)} \partial_{y} \wor(b,d)- \partial_{y}\wor(x,d)\right),
\end{equation}
where for $x>d$,
 $$\partial_{y}\wor(x,d):=\left(\lim_{z\to d} \frac{\wor(x,z)-\wor(x,d)}{z-d}\middle)\right|_{y=d}.$$
\end{cor}

\begin{proof}[Proof of Corollary \ref{cor:2}]
Similar to the discussion in \cite[Theorem 7]{Kyprianou2010:Rlevy}, the desired result is derived by letting $c\to d-$ in Corollary \ref{cor:1}.
From the right continuity of $X$, we have
\begin{align*}
&\  \Em_{x}\left[e^{-L(\kappa^{-}_{d})}; X(\kappa^{-}_{d})=d, \kappa^{-}_{d}\leq \kb\right]
 =\lim_{c\to d-} \Em_{x}\left[e^{-L(\ke)}; \ke\leq \kb\wedge\kc\right]\\
 = &\ \lim_{c\to d-} \frac{d-c}{\wor(d,c)} \frac{1}{d-c}\left( \wor(x,c)- \wor(x,d) \frac{\wor(b,c)}{\wor(b,d)}\right).
\end{align*}
On the other hand, from the discussion after Proposition \ref{prop:wor:wo}, it follows that   for $c<d\leq a$,
\[
\wor(d,c)=\wo(d,c)=W(d-c)+ \int_{c}^{d} W(d-z)\omega(z)\wo(z,c) dz.
\]
The mean value theorem yields $\wor(d,c)=W'(0+) (d-c)+ o(d-c)$. And
\begin{align*}
&\ \frac{1}{d-c}\left(\wor(x,c)- \wor(x,d)\frac{\wor(b,c)}{\wor(b,d)} \right)\\
=&\ \frac{\wor(x,c)\wor(b,d)- \wor(x,d) \wor(b,c)}{(d-c)\wor(b,d)}
\to \left(\wor(x,d) \frac{\partial_{y} \wor(b,d)}{\wor(b,d)}- \partial_{y}\wor(x,d)\right),
\end{align*}
by letting $c\to d-$. This finishes the proof.
\end{proof}

Noting that the right hand side of the formula is exactly the $y$-derivative of the density of $\Ror(x,dy)$ at $(x,y)$, which coincides with conclusion of  SNLP without refraction.

\section{Examples}\label{sec:4}
The refracted SNLP has been well-defined in \cite{Kyprianou2010:Rlevy},
but there are  not  many results yet on its occupation times as far as we know.
In this section, we apply our results to some simpler examples of $\omega$ to reproduce the existing results in \cite{Kyprianou2010:Rlevy, Kyprianou2013:occupationtime:Rlevy, Renaud2014:Rlevy} for which we essentially need to find expressions for $(\wor,\zor)$.

\begin{exam}\label{exam:1} Firstly, for the case $\omega(z)\equiv q\geq0$, we want to find $(\wor,\zor)$ such that
\[
\wor(x,y)=  w(x,y)+ q \int_{y}^{x} w(x,z)\wor(z,y) dz
\]
and
\[
\zor(x,y) = 1+ q \int_{y}^{x} w(x,z)\zor(z,y) dz.
\]
Actually, $\wo(x,y)=W^{(q)}(x-y)$ and $\zo(x,y)=Z^{(q)}(x-y)$, see \cite{BoLi:sub1} and we have
\[
\wor(x,y)= \wqr(x,y):= W^{(q)}(x-y)+ \delta \int_{a}^{x} \wz^{(q)}(x-z) W^{(q)}(dz-y),
\]
and
\[
\zor(x,y)= \zqr(x,y):= Z^{(q)}(x-y)+ \delta q \int_{a}^{x}
\wz^{(q)}(x-z) W^{(q)}(z-y) dz,
\]
by Proposition \ref{prop:wor:wo},
where $\wqr$ and $\zqr$ coincide with the functions of order $q$ used in  \cite[Theorem 4 and 6]{Kyprianou2010:Rlevy} and introduced in \cite{Renaud2014:Rlevy}.
\end{exam}

With the $(\wqr,\zqr)$ from above, applying Proposition \ref{prop:idens:2},
we have more identities as follows.
\begin{equation}\label{eqn:wzor:q}
\left\{
\begin{split}
\wor(x,y)=&\ \wqr(x,y)+ \int_{y}^{x}\wqr(x,z)(\omega(z)-q)\wor(z,y) dz\\
=&\ \wqr(x,y)+ \int_{y}^{x} \wor(x,z)(\omega(z)-q)\wqr(z,y) dz\\
\zor(x,y)=&\ \zqr(x,y)+ \int_{y}^{x}\wqr(x,z)(\omega(z)-q)\zor(z,y) dz\\
=&\ \zqr(x,y)+ \int_{y}^{x}\wor(x,z)(\omega(z)-q)\zqr(z,y) dz
\end{split}\right.
\quad\text{for $x,y\in\mathbb{R}$,}
\end{equation}
by taking $(\omega_{2},\omega_{1})=(\omega,q)$
and $(\omega_{2},\omega_{1})=(q,\omega)$ in the equations respectively.

In particular, for the case $\omega(x)=p+ (q-p) \boldsymbol{1}_{\{x<a\}}$, which is the function  considered in \cite{Kyprianou2013:occupationtime:Rlevy,Renaud2014:Rlevy},
we have for $x,y\in\mathbb{R}$
\begin{equation}\label{eqn:wqp}
\left\{
\begin{split}
\wor(x,y)
=&\ \wqr(x,y)- (q-p) \int_{a}^{x} \wor(x,z)\wqr(z,y)\,dz\\
=&\ \wqr(x,y)- (q-p) \int_{a}^{x} \wz^{(p)}(x-z)\wqr(z,y)\,dz,\\
\zor(x,c)
=&\ \zqr(x,c)- (q-p) \int_{a}^{x}\wz^{(p)}(x-z)\zqr(z,c)\,dz,
\end{split}\right.
\end{equation}
where the facts $q-\omega(z)=(q-p) \boldsymbol{1}_{\{z\geq a\}}$, $\omega(z)=p$ for $z\geq a$ and $\wor(x,z)=\wzo(x,z)=\wz^{(p)}(x-z)$ for $z\geq a$ after Proposition \ref{prop:wor:wo} are used,
and which gives \cite[Theorem 3]{Renaud2014:Rlevy}.

By letting the boundaries go to infinity in Theorem \ref{thm:1}, one
can solve the one-sided exit problems in principle. However, it is
not easy to carry it out when the behaviors of $\omega(\cdot)$
near infinity are arbitrary. Here, for a simpler case that
$\omega(x)$ is a constant for $|x|>M_{0}$ for some $M_{0}>0$, we
could work out the following Proposition.

\begin{prop}\label{prop:limits}
If $\omega(z)=p$ for all $z>M_{1}$ for some $p>0$ and $M_{1}\in\mathbb{R}$, we have as $x\to\infty$
\[
\frac{e^{\varphi(p)(y-x)}\wor(x,y)}{\varphi'(p)}\to
1
+ \int_{y}^{\infty} (\omega(u)-p)\wor(u,y)e^{\varphi(p)(y-u)}\,du
- \delta \int_{y-}^{a}e^{\varphi(p)(y-u)} \wo(du,y),
\]
and
\[
\frac{e^{\varphi(p)(y-x)}\zor(x,y)}{\varphi'(p)}\to
\frac{p}{\varphi(p)}+ \int_{y}^{\infty}
(\omega(u)-p)\zor(u,y)e^{\varphi(p)(y-u)}\,du - \delta \int_{y}^{a}
e^{\varphi(p)(y-u)} \zo(du,y).
\]
Similarly, if $\omega(z)=q$ for all $z<M_{2}$ for some $q>0$ and
$M_{2}\in\mathbb{R}$, we have as $y\to -\infty$
\begin{align*}
&\ \lim_{y\to-\infty}\frac{e^{\Phi(q)(y-x)} \wor(x,y)}{\Phi'(q)}
= \lim_{y\to-\infty} \frac{e^{\Phi(q)(y-x)} \zor(x,y)}{\Phi'(q)} \frac{\Phi(q)}{q}\\
=&\ 1+
\int_{\mathbb{R}} \wor(x,u)(\omega(u)-q) e^{\Phi(q)(u-x)}\,du
+ \delta \Phi(q)\int_{a}^{x} e^{\Phi(q)(u-x)} \wzo(x,u)\,du.
\end{align*}
\end{prop}

For the one-sided first passage problems,
which are studied in \cite{Kyprianou2010:Rlevy, Kyprianou2013:occupationtime:Rlevy, Renaud2014:Rlevy},
we  have the following more general results from Proposition \ref{prop:limits} and Theorem \ref{thm:1}.
\begin{cor}[One-sided first passage times]\label{cor:3}\
\begin{enumerate}
\item
If $\omega(z)=p>0$ for all $z>M_{1}$ for some
$M_{1}\in\mathbb{R}$, we have
\begin{align*}
&\  \Em_{x}\left(e^{-L(\kc)}; \kc<\infty\right)=
\zor(x,c)- \wor(x,c) \\
&\ \quad\times
\frac{\frac{p}{\varphi(p)}+
\int_{c}^{\infty} (\omega(u)-p)\zor(u,c)e^{\varphi(p)(c-u)}\,du
- \delta \int_{c}^{a} e^{\varphi(p)(c-u)} \zo(du,c)}
{1
+ \int_{c}^{\infty} (\omega(u)-p)\wor(u,c)e^{\varphi(p)(c-u)}\,du
- \delta \int_{c-}^{a}e^{\varphi(p)(c-u)} \wo(du,c)}.
\end{align*}
\item
If $\omega(z)=q>0$ for  all $z<M_{2}$ for some
$M_{2}\in\mathbb{R}$, we have
\begin{align*}
&\ \Em_{x}\left(e^{-L(\kb)}; \kb<\infty\right)\\
=&\ e^{\Phi(q)(x-b)}
\frac{1+ \int_{-\infty}^{x} \wor(x,u)(\omega(u)-q) e^{\Phi(q)(u-x)}\,du
+ \delta \Phi(q)\int_{a}^{x} e^{\Phi(q)(u-x)} \wzo(x,u)\,du}
{1+ \int_{-\infty}^{x} \wor(b,u)(\omega(u)-q) e^{\Phi(q)(u-b)}\,du
+ \delta \Phi(q)\int_{a}^{b} e^{\Phi(q)(u-b)} \wzo(b,u)\,du}.
\end{align*}
\end{enumerate}
\end{cor}


With \eqref{eqn:wqp} for $\omega(x)=p+ (q-p) \boldsymbol{1}_{\{x<a\}}$, we here only
compare the ratios in Corollary \ref{cor:3} with existing results
from \cite[Theorem 5 and 6]{Kyprianou2010:Rlevy} and \cite[Corollary
1 and 2]{Renaud2014:Rlevy} for $c=0$.

If $\omega(z)\equiv q=p$, then
\[\wo(x,y)=W^{(q)}(x-y), \quad
\zo(x,y)=Z^{(q)}(x-y), \quad \varphi(q)>\Phi(q).\]
and the following Laplace
transforms satisfy
\begin{equation}\label{eqn:laps}
\widehat{W^{(q)}}(\varphi(q))
=\frac{1}{\delta \varphi(q)} \quad\text{and}\quad
\widehat{dW^{(q)}}(\varphi(q))
:=\int_{0-}^{\infty} e^{-\varphi(q)x}W^{(q)}(dx)
=\frac{1}{\delta}.
\end{equation}
We thus have
\begin{align*}
&\  \frac{p}{\varphi(p)}+
\int_{c}^{\infty} (\omega(u)-p)\zor(u,c)e^{\varphi(p)(c-u)}\,du
- \delta \int_{c}^{a} e^{\varphi(p)(c-u)} \zo(du,c)\\
=&\ \frac{q}{\varphi(q)}- \delta q \int_{0}^{a} e^{-\varphi(q)u} W^{(q)}(u)\,du
= q \delta \int_{a}^{\infty} e^{-\varphi(q)u}W^{(q)}(u)\,du.
\end{align*}
Applying the Laplace transforms in \eqref{eqn:laps} one gives
\begin{align*}
&\ 1+ \int_{c}^{\infty} (\omega(u)-p)\wor(u,c)e^{\varphi(p)(c-u)}\,du
- \delta \int_{c-}^{a}e^{\varphi(p)(c-u)} \wo(du,c)\\
=&\ 1- \delta \int_{0-}^{a} e^{-\varphi(q)u} W^{(q)}(du)
= \delta \int_{a}^{\infty} e^{-\varphi(q)u} W^{(q)}(du) \\
\end{align*}
and
\begin{align*}
&\ 1+ \int_{-\infty}^{x} \wor(x,u)(\omega(u)-q) e^{\Phi(q)(u-x)}\,du
+ \delta \Phi(q)\int_{a}^{x} e^{\Phi(q)(u-x)} \wzo(x,u)\,du\\
=&\ 1+ \delta \Phi(q) \int_{a}^{x} e^{\Phi(q)(u-x)} \wz^{(q)}(x-u)\,du
\end{align*}
 which coincide with the formulas of the Laplace transforms in \cite[Theorem 5 and 6]{Kyprianou2010:Rlevy}.

If $\omega(z)= q 1_{\{z<a\}}$, $p=0$ and $\Em[Y_{1}]>\delta$, we
have
\[\varphi(0)=0, \quad
\frac{p}{\varphi(p)}=\psi_{Z}'(0)=\Em[Y_{1}]-\delta\] and
\[\omega(z)-p=q 1_{\{z<a\}}, \, \wor(z,c)=W^{(q)}(z), \,
\zor(z,c)=Z^{(q)}(z) \,\,\, \text{for} \,\,\, z<a.\]
Therefore,
\begin{align*}
&\ 1+ \int_{c}^{\infty} (\omega(u)-p)\wor(u,c)e^{\varphi(p)(c-u)}\,du
- \delta \int_{c-}^{a}e^{\varphi(p)(c-u)} \wo(du,c)\\
=&\ 1+ q \int_{0}^{a} W^{(q)}(z)\,dz- \delta \int_{0-}^{a} W^{(q)}(du)
= Z^{(q)}(a)- \delta W^{(q)}(a)
\end{align*}
and
\begin{align*}
&\ \frac{p}{\varphi(p)}+
\int_{c}^{\infty} (\omega(u)-p)\zor(u,c)e^{\varphi(p)(c-u)}\,du
- \delta \int_{c}^{a} e^{\varphi(p)(c-u)} \zo(du,c)\\
=&\ \psi_{Z}'(0)+ q \int_{0}^{a}Z^{(q)}(z)\,dz- \delta q \int_{0}^{a} W^{(q)}(u)\,du\\
=&\ \Em[Y_{1}]-\delta + q \int_{0}^{a} \left(\zz^{(q)}(y)- \delta \zz^{(q)}(a-y)W^{(q)}(y)\right)\,dy
\end{align*}
where the last equation comes from the integral over $[0,a]$ of
functions
\[ \zz^{(q)}(x)- Z^{(q)}(x)= \delta q \int_{0}^{x} \wz^{(q)}(x-z) W^{(q)}(z)\,dz,
\]
by taking $\omega(z)\equiv q$, $y=0$ in Proposition \ref{prop:wor:wo},
and they coincide with \cite[Corollary 1(i)]{Renaud2014:Rlevy}.
Similarly, $\omega(z)-q=-q 1_{\{z>a\}}$, then
\[\wor(x,z)=\wz^{(\omega)}(x,z)=\wz(x-z)\quad\text{for}\quad z>a.\]
We have
\begin{align*}
&\ 1+ \int_{-\infty}^{x} \wor(x,u)(\omega(u)-q) e^{\Phi(q)(u-x)}\,du
+ \delta \Phi(q)\int_{a}^{x} e^{\Phi(q)(u-x)} \wzo(x,u)\,du\\
=&\ 1- q \int_{a}^{x} e^{\Phi(q)(u-x)}\wz(x-u)\,du+
\delta \Phi(q) \int_{a}^{x} e^{\Phi(q)(u-x)}\wz(x-u)\,du\\
=&\ 1- (q-\delta \Phi(q)) \int_{0}^{x-a}\wz(y) e^{-\Phi(q) y}\,dy,
\end{align*}
which coincides with  \cite[Corollary 2(i)]{Renaud2014:Rlevy}.

For the end of this section,
we consider the case of $\omega(\cdot)$ being an $n$-step function.
Besides the relation given in Proportion \ref{prop:idens:2},
an inductive way is provided to define the function, similar to \cite{BoLi:sub1},

\begin{exam}\label{exam:2} Let $\lambda_{j}\geq0$ and $a_{n}> a_{n-1}>\cdots>a_{1}$ be
constants. Let \[\omega_{n}(x)=\lambda_{0}+ \sum_{j=1}^{n}
(\lambda_{j}-\lambda_{j-1}) 1_{\{x\geq a_{j}\}}\] be a step
function. Then $w^{(\omega_{n})}(x,y)$ and $z^{(\omega_{n})}(x,y)$
can be define inductively as follow. For $k\geq 1$,
\[
w^{(\omega_{k})}(x,y)
= w^{(\omega_{k-1})}(x,y)+ (\lambda_{k}-\lambda_{k-1}) \int_{a_{k}}^{x} w^{(\lambda_{k})}(x,z) w^{(\omega_{k-1})}(z,y) dz
\]
and for $a_{1}>y$,
\[
 z^{(\omega_{k})}(x,y) = z^{(\omega_{k-1})}(x,y)+
(\lambda_{k}-\lambda_{k-1}) \int_{a_{k}}^{x} w^{(\lambda_{k})}(x,z)
z^{(\omega_{k-1})}(z,y) dz
\]
with $w^{(\omega_{0})}(x,y)=w^{(\lambda_{0})}(x,y)$
and $z^{(\omega_{0})}(x,y)=z^{(\lambda_{0})}(x,y)$ defined in Example \ref{exam:1}.
\end{exam}
\begin{proof}[Proof of Example \ref{exam:2}]
Observing that $\omega_{k}(z)-\omega_{k-1}(z)=0$ for $z<a_{k}$, we have
\[
w^{(\omega_{k})}(x,y)
=w^{(\omega_{k-1})}(x,y) \quad \text{for $x<a_{k}$}.
\]
Since $\omega_{k}(z)-\lambda_{k}=0$ for $z\geq a_{k}$ and $w^{(\omega_{k})}(z,y)=0$ for $z>y$, we have from \eqref{eqn:wzor:q}
\begin{align*}
w^{(\omega_{k})}(x,y)
=&\ w^{(\lambda_{k})}(x,y)+ \int_{y}^{x} w^{(\lambda_{k})}(x,z)(\omega_{k}(z)-\lambda_{k}) w^{(\omega_{k})}(z,y) dz\\
=&\ w^{(\lambda_{k})}(x,y)+ \int_{y}^{a_{k}} w^{(\lambda_{k})}(x,z)(\omega_{k}(z)-\lambda_{k}) w^{(\omega_{k})}(z,y) dz\\
=&\ w^{(\lambda_{k})}(x,y)+ \int_{y}^{a_{k}} w^{(\lambda_{k})}(x,z)(\omega_{k-1}(z)-\lambda_{k}) w^{(\omega_{k-1})}(z,y) dz.
\end{align*}
Meanwhile, applying \eqref{eqn:wzor:q} to $w^{(\omega_{k-1})}(x,y)$
with $q=\lambda_{k}$ we have
\[
w^{(\omega_{k-1})}(x,y)=
w^{(\lambda_{k})}(x,y)+ \int_{y}^{x} w^{(\lambda_{k})}(x,z)(\omega_{k-1}(z)-\lambda_{k}) w^{(\omega_{k-1})}(z,y) dz.
\]
Inductive formula for $w^{(\omega_{k})}(x,y)$ is thus proved by
comparing the two identities above. Similar discussion could be
applied to derive formulas for $z^{(\omega_{k})}(x,y)$ but under the
condition $a_{1}\geq y$.
\end{proof}

\section{Proofs}\label{sec:5}
Before proving our main results, we comment on  the structure of this section.
Lemma \ref{lem:equation} is slightly postponed and the proof of Theorem \ref{thm:1} is split into two parts.
More specifically,
instead of showing the existence of solution to \eqref{eqn:lem} directly, using a quantity associated to  the event $\{\kb\leq \kc\}$ we first define  in \eqref{eqn:kb} an auxiliary function
$\wor(x,y)$, which
solves equation \eqref{defn:wor}
and is used in the expression for \eqref{thm:1a}.
We then establish the uniqueness of solution for Lemma \ref{lem:equation}
and find a solution to \eqref{eqn:lem}   using function
$\wor(\cdot,\cdot)$. After finishing the proof for Lemma \ref{lem:equation},
we continue with  proofs for the rest of results in Theorem \ref{thm:1}.

\begin{proof}[Proof of Theorem \ref{thm:1}(1)] We first focus on $\{\kb\leq\kc\}$.
To simplify the notation, we denote by
\[
\mathcal{A}(x;b):= \Em_{x}\left(\exp\left(-\int_{0}^{\kb}\omega(X_{t}) dt\right); \kb\leq \kc\right), \quad\text{for $c\leq x\leq b$}.
\]

We have from the absence of positive jumps and the Markov property for $X$ that
\begin{equation}\label{eqn:markov}
\mathcal{A}(x;z)=\mathcal{A}(x;y)\mathcal{A}(y;z)
\quad \text{for any $z>y>x>c$}.
\end{equation}

On the other hand, for every $t>0$, it holds that
\[
1-e^{-L(t)}=e^{-L(t)}\int_{0}^{t} e^{L(s)}\omega(X_{s}) ds=\int_{0}^{\infty} 1_{\{s<t\}}\omega(X_{s}) e^{-(L(t)-L(s))} ds.
\]
Applying  Fubini's theorem we  have
\begin{align*}
&\ \Em_{x}\left(1-e^{-L(\kb)}; \kb\leq \kc\right)\\
=&\ \int_{0}^{\infty} \Em_{x}\left(\omega(X_{s}) e^{-(L(\kb)-L(s))}; s<\kb\leq \kc\right) ds\\
=&\ \int_{0}^{\infty} \Em_{x}\left(\omega(X_{s}) 1_{\{s\leq \kb\wedge\kc\}}
\left(e^{-L(\kb-s)} 1_{\{ (\kb-s)\leq (\kc-s)\}}\right)\circ \theta_{s}\right) ds
\end{align*}
where $\theta_{s}$ is the shifting operator of $X$ such that $X_t\circ\theta_s=X_{s+t} $ for any $s, t\geq 0 $.
Applying the Markov property gives
\begin{align}
\mathcal{A}(x;b)=&\ \Pm_{x}(\kb\leq \kc)- \int_{c}^{b} \Rr(x,dy)\omega(y) \mathcal{A}(y;b)\label{eqn:FeyKac}\\
=&\ \frac{w(x,c)}{w(b,c)} \left(1- \int_{c}^{b}w(b,y)\omega(y)\mathcal{A}(y;b)\right) dy
+ \left(\int_{c}^{x} w(x,y)\omega(y) \mathcal{A}(y;x) dy\right)\mathcal{A}(x;b)
\non
\end{align}
where the identity \eqref{eqn:markov} is applied in the last identity for every $y\in(c,x)$. At the risk of abusing notation define
\begin{equation}\label{eqn:kb}
\wor(x,y):= w(x,y)\left/ \middle(1- \int_{y}^{x} w(x,z)\omega(z) \mathcal{A}(z;x) dz\right)
\quad\text{for $x>y\geq c$,}
\end{equation}
then $\mathcal{A}(x;b)=\frac{\wor(x,c)}{\wor(b,c)}$. Substituting this identity into \eqref{eqn:kb} again  gives
\[
\wor(x,y)= w(x,y)+ \int_{y}^{x} w(x,z)\omega(z) \wor(z,y) dz,
\]
which is  equation \eqref{defn:wor} and can be generalised to $\mathbb{R}\times\mathbb{R}$ naturally.
\end{proof}

\begin{proof}[Proof of Lemma \ref{lem:equation}]
Equation \eqref{eqn:lem} is a kind of renewal type equation.
Similar to the proof of \cite[Lemma 2.1]{BoLi:sub1}, we only need to focus on an arbitrary and fixed cylinder set $[c,b]\times[c,b]$ for the functions involved. Let $M_{3}\geq \sup_{x\in[c,b]}\omega(x)$ and $s_{0}>0$ such that $\widehat{\wz}(s_{0})\leq \frac{1}{2M_{3}}$.

\paragraph{Uniqueness}
To prove the uniqueness of solution of \eqref{eqn:lem} we show that, for fixed $y_{0}\in[c,b]$, $H^{(\omega)}(x,y_{0})=0$ is the only solution of
\[
H^{(\omega)}(x,y_{0})= \int_{y_{0}}^{x}w(x,z)\omega(z) H^{(\omega)}(z,y_{0}) dz.
\]
Actually, we have from \eqref{defn:wr:2} and our assumption that, $w(x,z)\leq \wz(x-z)$. Then
\begin{align*}
|e^{-s_{0} x} H^{(\omega)}(x,y_{0})|
\leq &\ \sup_{z\in[y_{0}, x]}|e^{-s_{0}z}H^{(\omega)}(z,y_{0})| \left( M_{3} \int_{y_{0}}^{x} e^{-s_{0}(x-z)}\wz(x-z) dz\right)\\
\leq&\  \frac{1}{2} \sup_{z\in[y_{0}, x]}|e^{-s_{0}z}H^{(\omega)}(z,y_{0})|,
\quad \text{ for any $x\in[c,b]$}.
\end{align*}
Thus $|e^{-s_{0} x} H^{(\omega)}(x,y_{0})|=0$  and $H^{(\omega)}(\cdot,y_{0})\equiv 0$.

\paragraph{Existence}
Now, $\wor(x,y)$ is well defined in \eqref{eqn:kb} and is the unique solution to
\[
\wor(x,y)= w(x,y)+ \int_{y}^{x} w(x,z)\omega(z) \wor(z,y) dz,
\]
from our previous discussion. For any $h(x,y)$,  define
\begin{equation}\label{eqn:Hw}
H^{(\omega)}(x,y):=h(x,y)+ \int_{y}^{x} \wor(x,z)\omega(z)h(z,y) dz.
\end{equation}
Then we have by change of variable and \eqref{defn:wor} that
\begin{align*}
&\ \int_{y}^{x} w(x,z)\omega(z) H^{(\omega)}(z,y) dz\\
=&\ \int_{y}^{x} w(x,z)\omega(z) h(z,y) dz + \int_{x>z>u>y} \left(w(x,z)\omega(z) \wor(z,u)\right)\omega(u)h(u,y) dz du\\
=&\ \int_{y}^{x} \wor(x,u) \omega(u) h(u,y) du= H^{(\omega)}(x,y)- h(x,y).
\end{align*}
 And this finishes the proof.
\end{proof}

\begin{rmk}
Identity \eqref{eqn:FeyKac} is essentially the Feyman-Kac formula in the context.
It can also be derived following the Poisson observation method in \cite{Zhou2014:occupationtime:levy, Zhou2015:occupationtime:levy, BoLi:sub1}.
\end{rmk}

With Lemma \ref{lem:equation} proved, one is free to use the results from Remark \ref{rmk:equation}.

\begin{rmk}\label{rmk:defn:zor}
Recalling the definition of $\zor$ in \eqref{defn:zor},
a conclusion from \eqref{eqn:Hw} with $h\equiv 1$ is that
\begin{equation}\label{defn:zor:2}
\zor(x,y)=1+ \int_{y}^{x}\wor(x,z)\omega(z) dz.
\end{equation}
And such defined $\zor(x,y)$ is the unique solution to \eqref{defn:zor}.
\end{rmk}

\begin{proof}[Proof of Theorem \ref{thm:1}(2)]
With $(\wor,\zor)$, we are ready to prove the rest of main results.

\paragraph{The resolvent $\Ror$}
For any bounded and measurable $f\geq0$,
it follows from \eqref{eqn:FeyKac}  that
\[
\Rr f(x)- \Ror f(x)= \int_{c}^{b} \Rr(x,dy) \omega(y) \Ror f(y).
\]
Thus, applying Proposition \ref{prop:Rlevy} to the equation above gives
\begin{align*}
\Ror f(x)=&\ w(x,c)\times c_{f}- \int_{c}^{x} w(x,y)f(y) dy
+ \int_{c}^{x} w(x,y) \omega(y) \Ror f(y) dy\\
=&\ \wor(x,c)\times c_{f}- \int_{c}^{x} \wor(x,y) f(y) dy\\
=&\ \int_{c}^{b} \left(\frac{\wor(x,c)}{\wor(b,c)}\wor(b,y)- \wor(x,y)\right)f(y) dy,
\end{align*}
where $c_{f}$ is a constant given by
\[
c_{f}=\int_{c}^{b} \frac{w(b,y)}{w(b,c)} \left(f(y)- \omega(y) V^{(\omega)}f(y)\right)\,dy,
\]
the second identity comes from Lemma \ref{lem:equation} and Remark \ref{rmk:equation},
and the last identity comes from the boundary condition that $\Ror f(b)=0$.

\paragraph{The quantity for $\{\kc\leq \kb\}$.} First notice that
\[
\Ror \left(\omega\right)(x)
= \Em_{x}\left(\int_{0}^{\kb\wedge\kc} e^{-L(t)}\omega(X_{t}) dt \right)
= 1- \Em_{x}\left(\exp\left(-\int_{0}^{\kb\wedge\kc}\omega(X_{s}) ds\right)\right).
\]
On the other hand, taking use of expression of the resolvent just obtained, we have for $x\in[c,b]$
\begin{align*}
\Ror \left(\omega\right)(x)=&\ \int_{c}^{b} \left(\frac{\wor(x,c)}{\wor(b,c)}\wor(b,y)- \wor(x,y)\right)\omega(y) dy\\
=&\ \frac{\wor(x,c)}{\wor(b,c)}\left(\zor(b,c)-1\right)- \left(\zor(x,c)-1\right),
\end{align*}
from Remark \ref{rmk:defn:zor}. Therefore, we have from the previous conclusion for event $\{\kb\leq\kc\}$ that
\begin{equation}
\Em_{x}\left(e^{-L(\kc)}; \kc\leq \kb\right)
=\zor(x,c)- \frac{\wor(x,c)}{\wor(b,c)}\zor(b,c).
\end{equation}
And this ends all the proof of Theorem \ref{thm:1}.
\end{proof}

For the proofs of Proposition \ref{prop:wor:wo} and \ref{prop:idens:2},
the fact $0=W(u-v)=w(u,v)=\wor(u,v)$ for $u<v$ is frequently used,
and we could rewrite \eqref{defn:wr:1} and \eqref{defn:wr:2} simply as
\begin{align}
w(x,y)=&\ W(x-y)+ \delta \int_{\mathbb{R}} 1_{\{z>a\}} \wz(x-z)W(dz-y)
\label{eqn:wr:12}\\
=&\ \wz(x-y)- \delta \int_{\mathbb{R}} 1_{\{z\leq a\}}\wz(x-z)W(dz-y).
\label{eqn:wr:22}
\end{align}

\begin{proof}[Proof of Proposition \ref{prop:wor:wo}]
To obtain identities between $\wor$ and $(\wo,\wzo)$,
denote  by 
\[
g^{(\omega)}(x,y):=W^{(\omega)}(x,y)+ \delta \int_{a}^{x} \wz^{(\omega)}(x,z)W^{(\omega)}(dz,y).
\]
the right-hand side of \eqref{eqn:wor:wo}.
Making use of \eqref{eqn:wr:22} for $w(x,y)$, we have for $x>y$
\begin{align*}
&\ \int_{y}^{x} w(x,z) \omega(z) g^{(\omega)}(z,y) dz\\
=&\ \int_{\mathbb{R}} \left(\wz(x-z)- \delta \int_{\mathbb{R}}
1_{\{u\leq a\}}\wz(x-u)W(du-z)\right) \omega(z)\\
&\quad \times \left(\wo(z,y)+ \delta \int_{\mathbb{R}} 1_{\{v>a\}}
\wzo(z,v) \wo(dv,y)\right) dz\\
=&\ \int_{\mathbb{R}} \wz(x-z)\omega(z)\wo(z,y) dz + \delta \iint_{v>a}
\left(\wz(x-z) \omega(z) \wzo(z,v) dz\right) \wo(dv,y)\\
&\quad - \delta \iint_{u\leq a} \wz(x-u) \left(W(du-z)\omega(z)\wo(z,y)\right) dz,
\end{align*}
where the fourth term after the first equality vanishes since
$1_{\{u\leq a\}} 1_{\{v>a\}} W(du-z)\wzo(z,v) \equiv 0\ \forall z\in\mathbb{R}$.
Further applying
\eqref{defn:wo} to $\wzo(x,y)$
and
\eqref{eqn:wzo:stj} to $\wo(dx,y)$,
the equation above equals to
\begin{align*}
&\ \int_{\mathbb{R}} \wz(x-z)\omega(z)\wo(z,y) dz
+\delta \int_{v>a} \left(\wzo(x,v)-\wz(x-v)\right) \wo(dv,y)\\
&\quad -\delta \int_{u\leq a} \wz(x-u) \left(\wo(du,y)- W(du-y)\right)\\
=&\ \int_{\mathbb{R}} \wz(x-z)\omega(z)\wo(z,y) dz
- \delta\int_{\mathbb{R}} \wz(x-u)\wo(du,y)\\
&\quad +\delta \left( \int_{v>a} \wzo(x,v)\wo(dv,y)+ \int_{u\leq a}\wz(x-u)W(du-y) \right).
\end{align*}
On the other hand, applying formula \eqref{eqn:wzo:stj}, \eqref{eqn:wzw} and \eqref{defn:wo}, we have
\begin{align*}
&\  \delta\int_{\mathbb{R}} \wz(x-u)\wo(du,y)\\
=&\  \delta\int_{\mathbb{R}} \wz(x-u)\left(W(du-y)+ \int_{\mathbb{R}} W(du-v)\omega(v) \wo(v,y) dv\right)\\
=&\ \wz(x-y)- W(x-y)+ \int_{\mathbb{R}} \left(\wz(x-v)- W(x-v)\right)\omega(v)\wo(v,y) dv\\
=&\ \wz(x-y)- \wo(x,y)+ \int_{\mathbb{R}} \wz(x-z)\omega(z)\wo(z,y) dz.
\end{align*}

Putting pieces together gives
\begin{align*}
&\ \int_{y}^{x} w(x,z)\omega(z) g^{(\omega)}(z,y) dz\\
=&\ \left(\wo(x,y)+ \delta \int_{a}^{x} \wzo(x,z)\wo(dz,y)\right)-
\left(\wz(x-y)- \delta \int_{y-}^{a} \wz(x,z)W(dz-y)\right)\\
=&\ g^{(\omega)}(x,y)- w(x,y)
\end{align*}
from the definition of $g^{(\omega)}(x,y)$ and \eqref{defn:wr:2}. It can be found that $\wor(x,y)$ and $g^{(\omega)}(x,y)$ satisfies the same equation. Thus $\wor(x,y)=g^{(\omega)}(x,y)$.

The second equation of \eqref{eqn:wor:wo} can be proved following the same idea by making use of \eqref{eqn:wr:12}.
\eqref{eqn:zor:zo} is a direct consequence of applying Remark \ref{rmk:defn:zor} to \eqref{eqn:wor:wo}.
And this finishes the proof.
\end{proof}

\begin{proof}[Proof of Proposition \ref{prop:idens:2}] To prove Proposition \ref{prop:idens:2}, we first claim that
\begin{equation}\label{defn:wor:2}
\wor(x,y)=w(x,y)+ \int_{y}^{x} \wor(x,z)\omega(z)w(z,y) dz.
\end{equation}

Denoting by $k(x,y)$ the right hand side of (\ref{prop:idens:2}),
by change of variable and  the order of integration we have
\begin{align*}
&\ \int_{\mathbb{R}} w(x,z)\omega(z)k(z,y) dz\\
=&\ \int_{\mathbb{R}} w(x,u)\omega(u)\left(w(u,y)+ \int_{\mathbb{R}} \wor(u,v)\omega(v)w(v,y) dv\right) du\\
=&\ \int_{\mathbb{R}} \left(w(x,v)+ \int_{\mathbb{R}} w(x,u)\omega(u)\wor(u,v) du\right)
\omega(v)w(v,y) dv\\
=&\ \int_{\mathbb{R}}\wor(x,v)\omega(v)w(v,y) dv= k(x,y)- w(x,y),
\end{align*}
by definition, and $k(x,y)=\wor(x,y)$ follows from the uniqueness of solution to \eqref{defn:wor}.

On the other hand, applying \eqref{defn:wor:2} to $w^{(\omega_{1})}$ and \eqref{defn:wor} to $w^{(\omega_{2})}$ twice in the following computation, we have for $x>y$
\begin{align*}
&\ \int_{\mathbb{R}}w^{(\omega_{1})}(x,z)\omega_{2}(z)w^{(\omega_{2})}(z,y) dz\\
=&\ \int_{\mathbb{R}}\left(w(x,z)+ \int_{\mathbb{R}}w^{(\omega_{1})}(x,u)\omega_{1}(u)w(u,z) du\right) \omega_{2}(z)w^{(\omega_{2})}(z,y) dz\\
=&\ w^{(\omega_{2})}(x,y)-w(x,y)+ \int_{\mathbb{R}}w^{(\omega_{1})}(x,u)\omega_{1}(u)\left(w^{(\omega_{2})}(u,y)-w(u,y)\right) du\\
=&\ w^{(\omega_{2})}(x,y)-w^{(\omega_{1})}(x,y)+ \int_{\mathbb{R}}w^{(\omega_{1})}(x,u)\omega_{1}(u)w^{(\omega_{2})}(u,y) du,
\end{align*}
which gives the desired equation.
The identity for $(z^{(\omega_{1})},z^{(\omega_{2})})$ can be proved by applying Remark \ref{rmk:defn:zor} to the identity of $(w^{(\omega_{1})}, w^{(\omega_{2})})$. And this finishes the proof.
\end{proof}

\begin{proof}[Proof of Proposition \ref{prop:limits}]
Firstly, let us consider the case of constant $\omega(\cdot)$.
Applying limit identities \eqref{eqn:limit:wz} to Proposition
\ref{prop:wor:wo} with $\omega(\cdot)=p$, one can check that, as
$x\to\infty$
\begin{align*}
 e^{\varphi(p)(y-x)} \wpr(x,y)
=&\ e^{\varphi(p)(y-x)}\left(\wz^{(p)}(x-y)-\delta \int_{y-}^{a}\wz^{(p)}(x-z)W^{(p)}(dz-y)\right)\non\\
\to&\ \varphi'(p)
\left(1- \delta\int_{y-}^{a} e^{\varphi(p)(y-z)}W^{(p)}(dz-y)\right)
\end{align*}
and
\begin{align*}
 e^{\varphi(p)(y-x)} \zpr(x,y)
=&\ e^{\varphi(p)(y-x)}\left(\zz^{(p)}(x-y)- \delta p \int_{y}^{a} \wz^{(p)}(x-z)W^{(p)}(z-y) dz\right)\non\\
\to&\ \varphi'(p)
\left(\frac{p}{\varphi(p)} - p \delta\int_{y}^{a} e^{\varphi(p)(y-z)}W^{(p)}(z-y) dz\right).
\end{align*}
Plugging them into \eqref{eqn:wzor:q},  with the fact of
$W^{(p)}(z-y)=0$ for $z<y$ we have
\begin{align*}
&\ \frac{e^{\varphi(p)(y-x)}\wor(x,y)}{\varphi'(p)}
=\frac{e^{\varphi(p)(y-x)}}{\varphi'(p)} \left(\wpr(x,y)+
\int_{\mathbb{R}}\wpr(x,z)(\omega(z)-p)\wor(z,y)\,dz\right)\\
\to&\ 1
- \delta \int_{\mathbb{R}} 1_{\{u\leq a\}}e^{\varphi(p)(y-u)}W^{(p)}(du-y)
+ \int_{\mathbb{R}} (\omega(z)-p)\wor(z,y) e^{\varphi(p)(y-z)}\,dz\\
&\quad - \delta \iint 1_{\{u\leq a\}}W^{(p)}(du-z)(\omega(z)-p)\wor(z,y) e^{\varphi(p)(y-u)}\,dz\\
=&\ 1
+ \int_{\mathbb{R}} (\omega(u)-p)\wor(u,y)e^{\varphi(p)(y-u)}\,du
- \delta \int_{y-}^{a}e^{\varphi(p)(y-u)} \wo(du,y),
\end{align*}
where for $u\leq a$ the fact $\wpr(u,y)= W^{(p)}(u-y)$,
$\wor(u,y)=\wo(u,y)$  and
the following Stieltjes measure from \eqref{eqn:wzor:q}  is applied in the last equation,
\[
\wor(du,y)= \wpr(du,y)+ \int_{\mathbb{R}} \wpr(du,z)(\omega(z)-p)\wor(z,y)\,dz.
\]

Similarly, as $x\to\infty$,
\begin{align*}
&\ \frac{e^{\varphi(p)(y-x)}\zor(x,y)}{\varphi'(p)}
=\frac{e^{\varphi(p)(y-x)}}{\varphi'(p)} \left(\zpr(x,y)+
\int_{z>y}\wpr(x,z)(\omega(z)-p)\zor(z,y)\,dz\right)\\
\to&\ \frac{p}{\varphi(p)}
- p \delta \int_{\mathbb{R}}  1_{\{z\leq a\}}
e^{\varphi(p)(y-z)}W^{(p)}(z-y) dz
+ \int_{z>y} (\omega(z)-p)\zor(z,y)e^{\varphi(p)(y-z)}\,dz\\
&\quad - \delta \iint W^{(p)}(du-z) 1_{\{u\leq a\}} 1_{\{z>y\}} (\omega(z)-p) \zor(z,y) e^{\varphi(p)(y-u)}\,dz \\
=&\ \frac{p}{\varphi(p)}+
\int_{u>y} (\omega(u)-p)\zor(u,y)e^{\varphi(p)(y-u)}\,du
- \delta \int_{y}^{a} e^{\varphi(p)(y-u)} \zo(du,y),
\end{align*}
with $\zor(u,y)=\zo(u,y)$, $\zpr(u,y)=Z^{(p)}(u-y)$ for $u\leq a$  and
\[
\zor(du,y)=p W^{(p)}(u-y) \,du+ \int_{z>y} W^{(p)}(du-z)(\omega(z)-p)\zor(z,y)\,dz
\]
is needed for the last equation.

Following the same procedure, as $y\to-\infty$ we have from \eqref{eqn:limit:wz}
\begin{align*}
e^{\Phi(q)(y-x)} \wqr(x,y)
=&\ e^{\Phi(q)(y-x)} \left( W^{(q)}(x-y)+ \delta \int_{a}^{x} \wz^{(q)}(x-z) W^{(q)}(dz-y)\right)\\
\to&\ \Phi'(q)\left(1+ \delta \Phi(q)\int_{a}^{x} e^{\Phi(q)(z-x)} \wz^{(q)}(x-z) \,dz\right).
\end{align*}
Plugging it into \eqref{eqn:wzor:q}, we have
\begin{align*}
&\ \frac{e^{\Phi(q)(y-x)}\wor(x,y)}{\Phi'(q)}
= \frac{e^{\Phi(q)(y-x)}}{\Phi'(q)} \left(\wqr(x,y)+ \int_{y}^{x}\wor(x,z)(\omega(z)-q) \wqr(z,y)\,dz\right)\\
\to&\ 1+ \delta \Phi(q) \int_{\mathbb{R}}1_{\{z>a\}} e^{\Phi(q)(z-x)} \wz^{(q)}(x-z)\,dz+ \int_{\mathbb{R}} \wor(x,z)(\omega(z)-q) e^{\Phi(q)(z-x)}\,dz\\
&\quad + \delta \Phi(q) \iint \wor(x,z)(\omega(z)-q)\wz^{(q)}(z-u) e^{\Phi(q)(u-x)} 1_{\{u>a\}}\,dudz\\
=&\ 1+ \delta \Phi(q)\int_{u>a} e^{\Phi(q)(u-x)} \wzo(x,u)\,du
+ \int_{\mathbb{R}} \wor(x,z)(\omega(z)-q) e^{\Phi(q)(z-x)}\,dz,
\end{align*}
where \eqref{eqn:wzor:q} and the facts
$\wor(x,u)=\wzo(x,u), \wqr(x,u)=\wz^{(q)}(x-u)$ for $u>a$
are applied in the last identity. One can also find
\[\D\lim_{y\to-\infty} \frac{e^{\Phi(q)(y-x)}\zor(x,y)}{\Phi'(q)}.\]
 This completes the proof.
\end{proof}

\paragraph{Acknowledgement} Bo Li is supported by National Natural Science Foundation of China (No.
11601243). Bo Li and Xiaowen Zhou are supported by NSERC (RGPIN-2016-06704).


\end{document}